\documentclass[a4paper]{amsart}
\usepackage{amssymb,latexsym}
\usepackage[utf8x]{inputenc}
\usepackage{amsmath,amsthm}
\usepackage{amsfonts,mathrsfs}
\usepackage{hyperref}
\usepackage{cleveref}
\usepackage{enumerate,units}
\usepackage[all]{xy}
\usepackage{graphicx,url}

\theoremstyle{plain}
\newtheorem{theorem}{Theorem}[section]
\newtheorem{maintheorem}[theorem]{Main Theorem}
\newtheorem{corollary}[theorem]{Corollary}
\newtheorem{lemma}[theorem]{Lemma}
\newtheorem{proposition}[theorem]{Proposition}

\newtheorem*{proposition*}{Proposition}

\newtheorem{fact}[theorem]{Fact}
\newtheorem{claim}{Claim}[theorem]

\newtheorem*{theorem*}{Theorem}
\newtheorem*{context*}{Assumption $\diamondsuit$}

\theoremstyle{definition}
\newtheorem{definition}[theorem]{Definition}

\newtheorem{example}[theorem]{Example}

\theoremstyle{remark}
\newtheorem{remark}[theorem]{Remark}

\newtheorem{question}[theorem]{\textbf{Question}}
\newtheorem{conjecture}[theorem]{\textbf{Conjecture}}

\numberwithin{equation}{section}
\newcommand{\forkindep}[1][]{%
  \mathrel{
    \mathop{
      \vcenter{
        \hbox{\oalign{\noalign{\kern-.3ex}\hfil$\vert$\hfil\cr
              \noalign{\kern-.7ex}
              $\smile$\cr\noalign{\kern-.3ex}}}
      }
    }\displaylimits_{#1}
  }
}

\newenvironment{claimproof}[1][\proofname]
  {%
    \proof[#1]%
  }
  {%
    \endproof%
  }

\newcounter{step}                   
    {\hfill $\clubsuit$             
     \vspace{7pt}\par}

\makeatletter
\providecommand*{\cupdot}{%
  \mathbin{%
    \mathpalette\@cupdot{}%
  }%
}
\newcommand*{\@cupdot}[2]{%
  \ooalign{%
    $\m@th#1\cup$\cr
    \hidewidth$\m@th#1\cdot$\hidewidth
  }%
}
\makeatother

\newcommand{\dom}{\mathrm{Dom}}

\newcommand{\Sh}{\text{Sh}}

\newcommand{\E}{\mathrel{E}}
\newcommand{\R}{\mathrel{R}}

\DeclareMathOperator{\dcl}{dcl} 
\DeclareMathOperator{\tp}{tp}

\DeclareMathOperator{\ch}{ch_T}
\DeclareMathOperator{\Sk}{Sk}

\begin{document}
\title{Infinite Cliques in Simple and Stable Graphs}

\author{Yatir Halevi}
\address{Department of Mathematics\\ University of Haifa\\ 199 Abba Khoushy Avenue \\ Haifa \\Israel}
 \email{ybenarih@campus.haifa.ac.il}

\author{Itay Kaplan}
\address{Einstein Institute of Mathematics, Hebrew University of Jerusalem, 91904, Jerusalem
Israel.}
\email{kaplan@math.huji.ac.il}

\author{Saharon Shelah}
\address{Einstein Institute of Mathematics, Hebrew University of Jerusalem, 91904, Jerusalem
Israel.}
\email{shelah@math.huji.ac.il}

\thanks{The first author would like to thanks the Israel Science Foundation for its support of this research (grant No. 555/21 and 290/19). The second author would like to thank the Israel Science Foundation for their support of this research (grant no. 1254/18). The third author would like to thank the Israel Science Foundation grants no. 1838/19 and  2320/23. This is Paper no. 1211 in the third author's  publication list.}

\keywords{chromatic number; stable graphs; simple graphs; infinite cliques; Taylor's Conjecture}
\subjclass[2020]{03C45; 05C15; 03C50}

\begin{abstract}
  Suppose that $G$ is a graph of cardinality $\mu^+$ with  chromatic number $\chi(G)\geq \mu^+$. One possible reason that this could happen is if $G$ contains a clique of size $\mu^+$. We prove that this is indeed the case when the edge relation is stable. When $G$ is a random graph (which is simple but not stable), this is not true. But still if in general the complete theory of $G$ is simple, $G$ must contain finite cliques of unbounded sizes.
\end{abstract}

\maketitle

\section{Introduction}
The chromatic number $\chi(G)$ of a graph $G$ is the minimal cardinal $\kappa$ for which there exists a vertex coloring with $\kappa$ colors such that connected vertices get different colors. 

Research around graphs having an uncountable chromatic number has a long history, see e.g., \cite[Section 3]{komjath}. This topic is set-theoretic in nature, with many results being independent of the axioms of set theory (ZFC). In \cite{1196,1211} we studied a specific conjecture (Taylor's strong conjecture) in the context of \emph{stable} graphs (in ZFC): a graph whose first order theory is stable (a model-theoretic notion of tameness, see Section \ref{sec:Preliminaries} for all the definitions). It turns out that a very close relative of this conjecture holds for stable graphs (although it does not hold in general).

More specifically, we showed that if a stable graph has chromatic number $>\beth_2(\aleph_0)$ then this implies the presence of all the finite subgraphs of a shift graph $\Sh_n(\omega)$ for some $0<n < \omega$, where for a cardinal $\kappa$, the shift graph $\Sh_n(\kappa)$ is the graph whose vertices are increasing $n$-tuples of ordinals in $\kappa$, and two such tuples $s,t$ are connected if for every $1\leq i\leq n-1$, $s(i)=t(i-1)$ or vice-versa (see Example \ref{E:shift}). In turn, this implies that the chromatic numbers of elementary extensions of said graph are unbounded. An important example for this paper is the case $n=1$: $\Sh_1(\kappa)$ is the complete graph on $\kappa$. For more on this result, see \cite{1196,1211}. 

In this paper, we take the first step towards identifying the $n$ in the previous paragraph, by considering not only the chromatic number of the graph, but its \emph{cardinality} as well, as we now explain.

Bounds for the chromatic number of the shift graph was computed by Erd\"os and Hajnal \cite{EH}: assuming the generalized continuum hypothesis (\emph{GCH}), 
\[\chi(\Sh_{n+1}(\kappa^{+n}))=\kappa\]
for all $n<\omega$, see Fact \ref{F:Sh-high chrom}. In particular (and trivially) $\chi(\Sh_1(\kappa))= \kappa$. Thus, it makes sense to ask the following question:

\begin{question} \label{Q:The question on Ch}
  Suppose that $G$ is a stable graph and for simplicity also GCH. Assume that for every cardinal $\kappa$, there is some $G' \equiv G$ (i.e., $\mathrm{Th}(G)=\mathrm{Th}(G')$) of cardinality $|G'| \leq \kappa^{+n}$ satisfying that $\chi(G') \geq \kappa^+$, then is it true that for some $m \leq n$, $G$ contains all finite subgraphs of $\Sh_{m}(\omega)$?
\end{question}

\begin{remark}
  Proposition \ref{P: stable no arbt large finite cliques} and Remark \ref{R:even stable theory} explain why we restricted ourselves to successor cardinals. 
\end{remark}

In this paper we deal with the case $n=1$, and we manage to give a satisfying solution to this case (and more) assuming that the theory of $G$ is simple or that the edge relation is stable (both weaker assumptions than stability of the theory). 
The following sums up the main results of this paper: 

\begin{maintheorem}[Propositions \ref{P:cannot embed infinite complete graph} and \ref{P:stable}]
  Let $G=(V,E)$ be a graph and $T=\mathrm{Th}(G,E)$ be its first order theory. Assume that $|G|=\mu^+$ and $\chi(G)\geq \mu^+$ for some infinite cardinal $\mu$.
  
  \begin{enumerate}
  \item Assuming $T$ is a simple theory then $G$ contains cliques of any finite size. 
  \item Assuming the edge relation $E$ is stable then $G$ contains an infinite clique of cardinality $\mu^+$.
  \end{enumerate}
  \end{maintheorem}

Note that the conclusion of item (1) (together with L\"owenheim-Skolem and compactness) implies the existence of $G'\succ G$ of cardinality $\mu^+$ that contains an infinite clique of cardinality $\mu^+$. The conclusion of item (2) is stronger: we can find such a clique already in $G$ itself.

\subsection{Organization of the paper} In Section \ref{sec:Preliminaries} we go over the relevant basic definitions in model theory and graph theory. Section \ref{sec:infinite cliques} contains the proof of the main results. Section \ref{S:chromatic} reframes the main results in terms of the function hinted on in Question \ref{Q:The question on Ch}. Finally, in the Appendix \ref{A:Hajnal-Komjath example} we analyze an example of Hajnal and Komj\'{a}th that was given by them to refute Taylor's strong conjecture and show that the theory of this graph is not stable (in fact, we show more: that it is not simple and has IP).

\section{Preliminaries}\label{sec:Preliminaries}

We use small latin letters $a,b,c$ for tuples and capital letters $A,B,C$ for sets. We also employ the standard model theoretic abuse of notation and write $a\in A$ even for tuples when the length of the tuple is immaterial or understood from context.

\subsection{Stability and Simplicity}
We use fairly standard model theoretic terminology and notation, see for example \cite{TZ}. We gather some of the needed notions. For stability, the reader may also consult with \cite{classification}.

We denote by $\tp(a/A)$ the complete type of $a$ over $A$. 
A structure $M$ is \emph{$\kappa$-saturated}, for a cardinal $\kappa$, if any type $p$ over $A$ with $|A|< \kappa$ is realized in $M$. The structure $M$ is \emph{saturated} if it is $|M|$-saturated.

The \emph{monster model} of a complete theory $T$, denoted here by $\mathbb{U}$, is a large saturated model containing all sets and models (as elementary substructures) we  encounter.\footnote{There are set theoretic issues in assuming that such a model exists, but these are overcome by standard techniques from set theory that ensure the generalized continuum hypothesis from some point on while fixing a fragment of the universe, see \cite{HaKa}. The reader can just accept this or alternatively assume that $\mathbb{U}$ is merely $\kappa$-saturated and $\kappa$-strongly homogeneous for large enough $\kappa$.}  All subsets and models will be \emph{small}, i.e. of cardinality $<|\mathbb{U}|$.

Given a first order theory $T$, a formula $\varphi(x,y)$ is $\emph{stable}$ if we cannot find elements $\langle a_i\in \mathbb{U}: i<\omega\rangle$ such that $\mathbb{U}\models \varphi(a_i,a_j)\iff i<j$. 

For any formula $\varphi(x,y)$ we set $\varphi(y,x)^{\text{opp}}=\varphi(x,y)$: it is the same formula but the roles of the variables replaced. The following is folklore.
\begin{fact}\label{F:def of types}
Work in a complete first order theory $T$ with infinite models. Let $\varphi(x,y)$ be a stable formula. There is a formula $\psi(y,z)$ such that  any $\varphi$-type $p$ over a model $M$ is definable by an instance of $\psi$ over $M$, i.e. for any such $\varphi$-type $p$ there is an element $c\in M$ such that $\varphi(x,b)\in p\iff M\vDash  \psi(b,c)$. Moreover, $\psi(y,z)$ can be chosen such that for any $c\in M$, $\psi(y,c)$ is equivalent to a boolean combination of instances of $\varphi^{\text{opp}}$. 
\end{fact}
\begin{proof}
By \cite[Theorem II.2.12]{classification}, there is some $\psi(y,z)$ such that for any $\varphi$-type $p$ over any set $A$ ($|A|\geq 2$), there is some $c_p\in A$ such that $\psi(y,c_p)$ defines $p$. By \cite[Lemma 2.2(i)]{Pillay}, if $p\in S_{\varphi}(M)$, where $M$ is a model, then $p$ is definable by a boolean combination of instances of $\varphi^{\text{opp}}$. But then as $M$ is a model $\psi(y,c_p)$ is equivalent to such a boolean combination.
\end{proof}

A theory $T$ is \emph{stable} if all formulas are stable.

Next we define simplicity; we give an equivalent definition, using the notion of dividing for types, see \cite[Proposition 7.2.5]{TZ}. Given a first order theory $T$ with a monster model $\mathbb{U}$, a formula $\varphi(x,b)$ with $b\in \mathbb{U}$ \emph{divides} over $A$ if there is a sequence of realizations $\langle b_i\in \mathbb{U}:i<\omega\rangle$ of $\tp(b/A)$ such that $\{\varphi(x,b_i):i<\omega\}$ is $k$-inconsistent for some $k<\omega$ (every set of size $k$ is inconsistent). A complete type $p$ over $B$ divides over $A$ if it contains some formula which divides over $A$.

The theory $T$ is \emph{simple} if for every complete type $p$ over $B$ there is some $A\subseteq B$ with $|A|\leq |T|$ such that $p$ does not divide over $A$. Every stable theory is simple.

The main tool we will use from simplicity theory is forking calculus. Non-forking independence is 3-place relation on sets (or tuples) denoted by $\forkindep$. We will not go over all the properties that non-forking independence enjoys in simple theories; see \cite[Chapter 7]{TZ} for more information.

\subsection{Graph theory}
Here we gather some facts on graphs and the chromatic number of graphs (see also \cite{1196}).

By a \emph{graph} we mean a pair $G=(V,E)$ where $E\subseteq V^2$ is symmetric and irreflexive. A \emph{graph homomorphism} between $G_1=(V_1,E_1)$ and $G_2=(V_2,E_2)$ is a map $f:V_1\to V_2$ such that $f(e)\in E_2$ for every $e\in E_1$. If $f$ is injective we will say that $f$ embeds $G_1$ into $G_2$ a subgraph. If in addition we require that $f(e)\in E_2$ if and only if $e\in E_1$ we will say that $f$ embeds $G_1$ into $G_2$ as an induced subgraph.

\begin{definition}
Let $G=(V,E)$ be a graph.
\begin{enumerate}
\item For a cardinal $\kappa$, a \emph{vertex coloring} (or just coloring) of cardinality $\kappa$ is a function $c:V\to \kappa$ such that $x\E y$ implies $c(x)\neq c(y)$ for all $x,y\in V$.
\item The \emph{chromatic number} $\chi(G)$ is the minimal cardinality of a vertex coloring of $G$.
\end{enumerate}
\end{definition}

The following is easy and well known.

\begin{fact}\label{F:basic-prop-chi}
Let $G=(V,\E)$ be a graph.
If $V=\bigcup_{i\in I} V_i$ then $\chi(G)\leq \sum_{i\in I} \chi(V_i, E\restriction V_i)$.
\end{fact}
\begin{proof}
Let $c_i:V_i\to \mu_i$ be a coloring of $(V_i,E\restriction V_i)$. Define a coloring $c:V\to \bigcup \{\mu_i\times\{i\}:i\in I\}$ by choosing for any $v\in V$ an $i_v\in I$ such that $v\in V_{i_v}$ and setting $c(v)=(c_{i_v}(v),i_v)$.
\end{proof}

\begin{example} \label{E:shift}
For any finite number $r\geq 1$ and any linearly ordered set $(A,<)$, let $\Sh_r(A)$, (\emph{the shift graph on $A$}) be the following graph: its set of vertices is the set of strictly increasing $r$-tuples from $A$, $\langle s_0<\dots<s_{r-1}\rangle $, and we put an edge between $s$ and $t$ if for every $1\leq i\leq r-1$, $s_i=t_{i-1}$, or vice-versa. It is an easy exercise to show that $\Sh_r(A)$ is a connected graph. If $r=1$ this gives $K_{A}$, the complete graph on $A$.
\end{example}

\begin{example}[Symmetric Shift Graph]
Let $r\geq 1$ be any natural number and $A$ any set. The \emph{symmetric shift graph} $\Sh_r^{sym}(A)$ is defined similarly as the shift graph but with set of vertices the set of distinct $r$-tuples.
Note that $\Sh_r(A)$ is an induced subgraph of $\Sh_r^{sym}(A)$ (and that for $r=1$ they are both the complete graph on $A$).

Since for any infinite set $A$, $\Sh_r^{sym}(A)$ is definable in $(A,=)$, it is stable. Moreover, for any two infinite sets $A$ and $B$, $\Sh_r^{sym}(A)\equiv \Sh_r^{sym}(B)$. Since every infinite set $A$ is saturated, $\Sh_r^{sym}(A)$ is saturated.
\end{example}

\begin{fact}\cite[Fact 2.6]{1196}\cite[Proof of Theorem 2]{EH-shift}\label{F:Sh-high chrom}
Let $1\leq r<\omega$ be a natural number and $\mu$ be an infinite cardinal,  \[\chi\left(\Sh_r^{sym}(\beth_{r-1}\left(\mu\right))\right)\leq\mu\] and \[\chi\left(\Sh_r(\beth_{r-1}\left(\mu\right)^{+})\right)\geq\mu^{+}.\]
\end{fact}

\begin{remark}
Note that it follows that $\chi(\Sh_{r}(\beth_{r-1}(\mu)))\leq \mu$.
\end{remark}

\begin{lemma}\label{L:sat of sym}
If $A$ is infinite then  $\Sh_r^{sym}(A)$ is a saturated model of $\mathrm{Th}(\Sh_r^{sym}(\omega))$ of cardinality $|A|$.
\end{lemma}
\begin{proof}
This follows easily from the fact that $(A,=)$ is saturated and that $\Sh_r^{sym}(A)$ is definable in $(A,=)$.
\end{proof}

We prove two easy results on the theory of the shift graphs.
\begin{lemma}\label{L:embed model of shift into symmetric}
Every model of  $T=\mathrm{Th}(\Sh_r(\omega))$ (of cardinality $\lambda$) can be embedded as an induced subgraph of $\Sh_r^{sym}(A)$ for some infinite set $A$ (of cardinality $\lambda$). 
\end{lemma}
\begin{proof}
Since $\Sh_r(\omega)$ is an induced subgraph of $\Sh_r^{sym}(\omega)$, the former satisfies the universal theory of $\Sh_r^{sym}(\omega)$. Thus every model $M$ of $T$ can be embedded as an induced subgraph of a model of $\mathrm{Th}(\Sh_r^{sym}(\omega))$. By Lemma \ref{L:sat of sym}, and the universality of saturated models, $M$ can be embedded as an induced subgraph of $\Sh_r^{sym}(A)$ for some infinite set $A$ of cardinality $|M|$.
\end{proof}

\begin{lemma}\label{L:properties of shift2}
For any cardinal $\mu$, the following holds for the graph $(\Sh_2(\mu),E)$:
\begin{enumerate}
\item Its complete theory is not stable.
\item Its graph relation is stable.
\item It is triangle-free.
\end{enumerate}
\end{lemma}
\begin{proof}
(1) For any $1<n<\omega$, let $X_n$ be the definable set $\{x\in \Sh_2(\mu): (x\E (0,n)) \wedge \neg(x\E (0,1))\}$, it is easily seen that $X_n=\{(n,k): k>n\}$. For any $0<m<\omega$, let $Y_m=\{x\in \Sh_2(\mu): (x\E (m,m+1))\wedge \neg(x\E (m-1,m+1))\}$, it is easily seen that $Y_m=\{(k,m): k<m\}$. 

For any pair of natural numbers $(n,m)$, with $n>1$ and $n<m$, let $\psi_{n,m}(x,y)$ be the formula (with parameters $n,m$): $\exists z (x\in X_n\wedge y\in Y_m\wedge (z\E x)\wedge (z\E y))$. 

For any $l$ with $l<m-n$ let $a_l=(n,n+l)\in X_n$, and for any $0<k<m-n$ let $b_k=(n+k,m)\in Y_m$. It is easily checked that $\psi_{n,m}(a_l,b_k) \iff l<k$. Since all the $\psi_{n,m}$'s are uniformly definable with parameters, by compactness we get that the theory of $\Sh_2(\mu)$ is not stable.

(2) One can either see directly that the graph relation is stable, or note that since $\Sh_2(\mu)$ is an induced subgraph of the stable graph $\Sh_2^{sym}(\mu)$ its edge relation must also be stable.

(3) Suppose that $(a,b)\E(c,d)\E(e,f)\E(a,b)$. Note that also $(c,d)\E(a,b)\E(e,f)\E(c,d)$. Hence, without loss of generality $b=c$. It follows easily that  $e=d$. So $a<b<d<f$. Now either $a=f$ or $b=e=d$, so either way we get a contradiction.
\end{proof}

\begin{remark}\label{R:Sh3}
\begin{enumerate} 
\item Since $\Sh_2(\mu)$ is definable in $(\mu,<)$, it has NIP (not the Independence Property (IP), see e.g., \cite{guidetonip}). 
\item One can also note that $\Sh_3^{sym}(\mu)$ is triangle-free.
\end{enumerate}
\end{remark}

Every graph is a (not necessarily induced) subgraph of a stable graph (e.g. a large enough complete graph). On the other hand, every shift graph $\Sh_n(\mu)$ is an induced subgraph of a stable graph (the symmetric shift graph). This raises the following:

\begin{question}
Is every graph with a stable edge relation an induced subgraph of a stable graph?
\end{question}

\section{Infinite Cliques}\label{sec:infinite cliques}
In this section we prove the main results of the paper. 

%

\subsection{Simple Graphs}
We start with the following technical result.

\begin{lemma}\label{L:technical simple lemma}
Let $M$ be some structure, in a language $L$, and assume that $T=\mathrm{Th}(M)$ is simple. Let $M'$ be some expansion of $M$ to a language $L'\supseteq L$. Let $<$ be a linear order on the universe of $M$ and $\alpha,\beta\in M$ elements satisfying that
\begin{enumerate}
\item $(\dcl_{L'}(\alpha),<)$ and $(\dcl_{L'}(\beta),<)$ are well-orders and
\item  for any $\gamma\in \dcl_{L'}(\alpha)\cup \dcl_{L'}(\beta)$, $\gamma \forkindep[\dcl_{L'}(\gamma)\cap \{\varepsilon\in M:\varepsilon<\gamma\}]\{\varepsilon\in M:\varepsilon<\gamma\}$.
\end{enumerate}

Then  \[\dcl_{L'}(\alpha)\forkindep[\dcl_{L'}(\alpha)\cap \dcl_{L'}(\beta)]\dcl_{L'}(\beta).\]
\end{lemma}
\begin{proof}
Let $\alpha,\beta\in M$ be elements as in the statement, and let $\mathbf{a}^\alpha=\dcl_{L'}(\alpha)$ and $\mathbf{a}^\beta=\dcl_{L'}(\beta)$.
Set $\Omega=\mathbf{a}^\alpha\cup\mathbf{a}^\beta$ and for any $\gamma\in \Omega$ let $\Omega_{<\gamma}=\{\varepsilon \in \Omega: \varepsilon<\gamma\}$. Since $(\Omega,<)$ is a finite union of well ordered sets, it is also well ordered. 

\begin{claim}\label{C:induction}
For $\gamma\in \Omega$, if \[\mathbf{a}^\alpha\cap \Omega_{<\gamma} \forkindep[\mathbf{a}^\alpha\cap \mathbf{a}^\beta\cap \Omega_{<\gamma}]\mathbf{a}^\beta\cap\ \Omega_{<\gamma}\]
then
\[\mathbf{a}^\alpha\cap \Omega_{\leq \gamma} \forkindep[\mathbf{a}^\alpha\cap \mathbf{a}^\beta\cap \Omega_{\leq \gamma}]\mathbf{a}^\beta\cap\ \Omega_{\leq \gamma}.\]
\end{claim}
\begin{claimproof}
By symmetry, we deal with the case $\gamma\in \mathbf{a}^\alpha$. We first prove that
\begin{equation}\label{eq:indep}
\mathbf{a}^\alpha\cap \Omega_{\leq \gamma}\forkindep[\mathbf{a}^\alpha\cap \mathbf{a}^\beta\cap \Omega_{<\gamma}] \mathbf{a}^\beta\cap \Omega_{<\gamma}.
\end{equation}

By hypothesis (2), $\gamma\forkindep[\dcl_{L'}(\gamma)\cap \{\varepsilon:\varepsilon<\gamma\}]\{\varepsilon:\varepsilon<\gamma\}$.
Since \[\dcl_{L'}(\gamma)\cap\{\varepsilon:\varepsilon<\gamma\}\subseteq \mathbf{a}^\alpha\cap \Omega_{<\gamma}\subseteq \{\varepsilon:\varepsilon<\gamma\},\] we get that $\gamma\forkindep[\mathbf{a}^\alpha\cap \Omega_{<\gamma}]\{\varepsilon:\varepsilon<\gamma\}$, so $\gamma\forkindep[\mathbf{a}^\alpha\cap \Omega_{<\gamma}] \mathbf{a}^\beta\cap \Omega_{<\gamma}$ and \[\mathbf{a}^\alpha\cap\Omega_{\leq \gamma}\forkindep[\mathbf{a}^\alpha\cap \Omega_{<\gamma}] \mathbf{a}^\beta\cap \Omega_{<\gamma}.\] By assumption and transitivity, we conclude (\ref{eq:indep}).

If $\gamma\notin \mathbf{a}^\beta$ then we are done; so assume that $\gamma\in \mathbf{a}^\beta$ as well. In this case, by properties of forking we get that $\mathbf{a}^\alpha\cap  \Omega_{\leq \gamma}\forkindep[\mathbf{a}^\alpha\cap \mathbf{a}^\beta\cap \Omega_{\leq \gamma}]\mathbf{a}^\beta\cap\Omega_{\leq \gamma}$, which is what we wanted to prove.
\end{claimproof}
The proof now follows by induction: the successor step is Claim \ref{C:induction} and limit case follows since forking is witnessed by a formula.
%
%
%
\end{proof}

We move to the main result of this section on simplicity.

\begin{proposition}\label{P:cannot embed infinite complete graph}
Let $T$ be a complete simple theory of graphs in the language of graphs $L=\{E\}$.

Let $\mu$ be an infinite cardinal and $G=(V,E)\vDash T$ with $|G|\leq 2^\mu$. If $\chi(G)\geq \mu^+$ then then there exists $G\equiv G'$ with $|G'|=\mu^+$ that contains a clique of cardinality $\mu^+$. 
\end{proposition}
\begin{proof}
By compactness and L\"owenheim-Skolem, it is sufficient to show that $G$ contains arbitrary large finite cliques. Furthermore, by passing to an elementary extension, we may assume that $|G|=2^\mu$. Additionally, after renaming elements, we may assume that $V=2^\mu$. Assume that $\chi(G)\geq \mu^+$.

As $T$ is simple, for every nonzero $\alpha\in V$,  $\tp(\alpha/\{\beta:\beta<\alpha\})$ does not fork over some nonempty countable subset $A_\alpha\subseteq \{\beta:\beta<\alpha\}$. Enumerate $A_\alpha$ as $\langle c_{\alpha,n}:n<\omega\rangle$, possibly with repetitions. Let $F_n$ be the function mapping a nonzero $\alpha\in V$ to $c_{\alpha,n}$ (and $F_n(0)=0$). 

Let $L' \supseteq L\cup \{F_n:n<\omega\}\cup \{<\}$ be a language containing Skolem functions and $G'$ an expansion of $G$ to $L'$ with Skolem functions such that the $F^n$'s are interpreted as above. Let $T'=\mathrm{Th}(G')$. As usual $\Sk(-)$ denotes the Skolem hull in $L'$, i.e. $\Sk(C)$ is the structure generated by $C$ in $G'$. Unless specified otherwise, whatever is done below is done in $L$.  

By forking base monotonicity and the choice of the functions $F_n$, for any $\alpha\in V$, $\tp(\alpha/\{\beta:\beta<\alpha\})$ does not fork over $\Sk(\{\alpha\})\cap \{\beta:\beta<\alpha\}$.
%

Let $\Delta$ be the collection of all of formulas in one variable $x$ over $\emptyset$ in $L'$ and let $\Delta=\bigcup_{l<\omega}\Delta_l$ be an increasing union of finite subsets. Let $\langle t_i(x): i<\omega\rangle$ be some enumeration of all the terms in $L'$, with $t_0(x)=x$. Thus $\Sk(\{ \alpha\})=\{t_i(\alpha):i<\omega\}$.

Enumerate the $2^\mu$ functions from $\mu$ to $\{0,1\}$ by  $\langle \eta_\alpha\rangle_{\alpha<2^\mu}$, without repetitions.

For any finite subset $u\subseteq \mu$ and $n<\omega$ we define a relation $R_{u,n}$ on 

\begin{equation*}
\begin{aligned}
\dom(R_{u,n}):=\{\alpha\in V=2^\mu: \eta_{t_i(\alpha)}\restriction u\neq \eta_{t_j(\alpha)}\restriction u \text{ for}\\ \text{all $i<j<n$ such that $t_i(\alpha)\neq t_j(\alpha)$}\}. 
\end{aligned}
\end{equation*}

Let $\alpha \R_{u,n} \beta$ (for $\alpha,\beta\in \dom(R_{u,n})$) if:
\begin{enumerate}
\item for all $j<n$, $\eta_{t_j(\alpha)}\restriction u=\eta_{t_j(\beta)}\restriction u$ and
\item $\tp_{\Delta_n}(\alpha)=\tp_{\Delta_n}(\beta)$.
\end{enumerate}
Note that $R_{u,n}$ is an equivalence relation on $\dom(R_{u,n})$.

\begin{claim}\label{C:exists alpha}
There exists $\alpha\in V$ such that for every finite subset $u\subseteq  \mu$ and $n<\omega$, if $\alpha\in \dom(R_{u,n})$ then 
\[\exists \beta\in [\alpha]_{R_{u,n}} (\alpha\E\beta).\]
\end{claim}
\begin{claimproof}
Note that each $R_{u,n}$ has only finitely many classes on $\dom(R_{u,n})$. Assume towards a contradiction that for every $\alpha\in V$ we can find some $u_\alpha\subseteq \mu$ and $n_\alpha<\omega$ that satisfy the negation of the statement. Now map $\alpha$ to $(u_\alpha,n_\alpha,[\alpha]_{R_{u_\alpha,n_\alpha}})$. This is easily a legal coloring of $V$ by $\mu$ colors, contradicting $\chi(G)\geq \mu^+$.
\end{claimproof}

Let $\alpha\in V$ be as supplied by Claim \ref{C:exists alpha}. For any $n<\omega$ let \[u_n=\{\min\{\varepsilon<\mu: \eta_{t_i(\alpha)}(\varepsilon)\neq \eta_{t_j(\alpha)}(\varepsilon)\}:i<j<n \text{ such that $t_i(\alpha)\neq t_j(\alpha)$}\};\] it is a finite subset of $\mu$. Easily, $\alpha\in \dom(R_{u_n,n})$ for all $n<\omega$. For $n<\omega$ let $\beta_n$ be the element given by Claim \ref{C:exists alpha} for $(u_n,n)$ (and $\alpha$).

Let  $\mathcal{U}$ be a nonprincipal ultrafilter on $\omega$ and let $\widetilde G=(G')^\omega/\mathcal{U}$ be the corresponding ultrapower in the language $L'$. We let $\widetilde G=(\widetilde V,E)$. Set $\widetilde \alpha=[\alpha]_{\mathcal{U}}$, $\widetilde \beta=[\beta_n]_{\mathcal{U}}$; so $\widetilde \alpha \E \widetilde \beta$. 
We make some observations. By definition of $\widetilde \alpha$, $\tp_{L'}(\widetilde \alpha)= \tp_{L'}(\alpha)$ and by definition of the relations $R_{u_n,n}$, this type is also equal to $\tp_{L'}(\widetilde \beta)$.

Set $\mathbf{a}^{\widetilde \alpha}=\langle t_i(\widetilde \alpha):i<\omega\rangle$ (and likewise for $\alpha$), $\mathbf{a}^{\widetilde \beta}=\langle t_i(\widetilde \beta):i<\omega\rangle$.

\begin{claim}\label{C: enumerate me}
$\mathbf{a}^{\widetilde \alpha}\cap \mathbf{a}^{\widetilde \beta}=\langle t_i(\widetilde \alpha):t_i(\widetilde \alpha)=t_i(\widetilde \beta)\rangle$.
\end{claim}
\begin{claimproof}
Note that for all $i<j<\omega$, if $t_i(\widetilde \alpha)=t_j(\widetilde \alpha)$ then the same holds for $\widetilde \beta$.

We claim that $t_i(\widetilde \alpha)=t_j(\widetilde \beta) \implies t_i(\widetilde \alpha)=t_i(\widetilde \beta)$. Otherwise, assume $i<j$. If $t_i(\widetilde \alpha)=t_j(\widetilde \alpha)$ then, by the first paragraph we are done. So assume not.  We can find some $n$ large enough for which $t_i(\beta_n)\neq t_j(\beta_n)$, $t_i(\alpha)=t_j(\beta_n)$ and $i< j<n$.

As $\beta_n\R_{u_n,n} \alpha$, $\eta_{t_i(\beta_n)}\restriction u_n=\eta_{t_i(\alpha)}\restriction u_n$ so by assumption $\eta_{t_i(\beta_n)}\restriction u_n=\eta_{t_j(\beta_n)}\restriction u_n$, contradicting the definition of $\dom(R_{u_n,n})$.
\end{claimproof}

As $(\mathbf{a}^\alpha,<)$ is a well-order (as a substructure of $2^\mu$) so are $(\mathbf{a}^{\widetilde \alpha},<)$ and $(\mathbf{a}^{\widetilde \beta},<)$. With the aim of applying Lemma \ref{L:technical simple lemma} with $M=\widetilde G$, we prove the following.

\begin{claim}\label{C:forking in tilde G}
For any  $\gamma\in \mathbf{a}^{\widetilde \alpha}\cup\mathbf{a}^{\widetilde \beta}$, $\tp_{L}(\gamma/\{\varepsilon\in \widetilde G :\varepsilon<\gamma\})$ does not fork over $\Sk(\gamma)\cap\{\varepsilon\in\widetilde G:\varepsilon<\gamma\}$.
\end{claim}
\begin{claimproof}
Assume that $\gamma\in \mathbf{a}^{\widetilde \alpha}$. The proof will only use the fact that $\tp_{L'}(\widetilde \alpha)=\tp_{L'}(\alpha)$. Hence, the same proof will also work for $\gamma\in \mathbf{a}^{\widetilde \beta}$. Recall that  $\tp_{L'}(\widetilde{\alpha})=\tp_{L'}(\widetilde \beta)$. Let $t(x)$ be a term (in $L'$) for which $\gamma=t(\widetilde \alpha)$, we get that for $\gamma':=t(\alpha)$, $\tp_{L'}(\gamma)=\tp_{L}(\gamma')$.
%

If $\tp_{L}(\gamma/\{\varepsilon :\varepsilon<\gamma\})$ forks over $\Sk(\gamma)\cap\{\varepsilon:\varepsilon<\gamma\}$ then by symmetry of forking there is a tuple $c$ of elements from $\{\varepsilon:\varepsilon <\gamma\}$ such that $\tp_L(c/\{\gamma\}\cup (\Sk(\gamma)\cap \{\varepsilon:\varepsilon<\gamma\}) )$ forks over $\Sk(\gamma)\cap \{\varepsilon:\varepsilon<\gamma\}$. Let $\varphi(y,z)$ be a formula over $\Sk(\gamma)\cap \{\varepsilon:\varepsilon<\gamma\}$, satisfied by $(c,\gamma)$, such that $\varphi(y,\gamma)$  forks over $\Sk(\gamma)\cap \{\varepsilon:\varepsilon<\gamma\}$. Since $\tp_{L'}(\gamma)=\tp_{L}(\gamma')$, $\varphi(y,\gamma')$ forks over $\Sk(\{\gamma'\})\cap \{\varepsilon:\varepsilon<\gamma'\}$.

On the other hand, we know that $\exists y<\gamma \varphi(y,\gamma)$ is in $\tp_{L'}(\gamma/\Sk(\gamma)\cap \{\varepsilon:\varepsilon<\gamma\})$ so $\exists y<\gamma' \varphi(y,\gamma')$ is in $\tp_{L'}(\gamma'/\Sk(\{\gamma'\})\cap \{\varepsilon:\varepsilon<\gamma'\})$. Consequently, there exists a tuple $c'$ of elements in $\{\varepsilon\in V:\varepsilon<\gamma'\}$ for which $\varphi(c',\gamma')$ holds, contradicting the fact that $\tp_L(\gamma'/\{\varepsilon:\varepsilon<\gamma'\})$ does not fork over $\Sk(\{\gamma'\})\cap \{\varepsilon:\varepsilon<\gamma'\}$ (by symmetry of non-forking).
\end{claimproof}

Recall that $\tp_{L'}(\widetilde \alpha)=\tp_{L'}(\widetilde \beta)$. Also, $\mathbf{a}^{\widetilde \alpha}$ and $\mathbf{a}^{\widetilde \beta}$ enumerate elementary substructures  of $\widetilde G\restriction L$. Setting $\mathbf{a}^0=\mathbf{a}^{\widetilde \alpha}\cap \mathbf{a}^{\widetilde\beta}=\langle t_i(\widetilde \alpha):t_i(\widetilde \alpha)=t_i(\widetilde \beta)\rangle$ by Claim \ref{C: enumerate me}, we have $\tp_L(\mathbf{a}^{\widetilde \alpha}/\mathbf{a}^0)=\tp_L(\mathbf{a}^{\widetilde \beta}/\mathbf{a}^0)$. Note that $\mathbf{a}^0$ is an elementary substructure as well.
%
%

By Claim \ref{C:forking in tilde G} and Lemma \ref{L:technical simple lemma}, $\mathbf{a}^{\widetilde \alpha}\forkindep[\mathbf{a}^0]\mathbf{a}^{\widetilde \beta}$. So by the independence theorem for simple theories, see  \cite[Lemma 7.4.8]{TZ}, we may find an indiscernible sequence starting with $\mathbf{a}^{\widetilde \alpha}$ and $\mathbf{a}^{\widetilde \beta}$ (in some elementary extension). Thus as $\widetilde \alpha\mathrel{E} \widetilde \beta$ we can find an infinite clique in that elementary extension.
\end{proof}

Given Proposition \ref{P:cannot embed infinite complete graph}, a natural question to ask if we can necessarily find an infinite clique already in $G$ it self. The following example shows that you cannot hope for an uncountable clique in general.

Recall that every random graph has a simple theory and that every such graph contains a countable infinite clique so these cannot be avoided.

\begin{proposition}\label{P:embed into random}
Let $\mu$ be an infinite cardinal. Any graph of cardinality $\mu$ without an uncountable clique can be embedded in a random graph of cardinality $\mu$ with no uncountable clique.
%
\end{proposition}
\begin{proof}
Let $\mu$ be an infinite cardinal and let $G_0=(V_0,E_0)$ be the given graph. Without loss of generality, $V_0=\{2\alpha: \alpha<\mu\}$.


Let $\langle u_\gamma\rangle_{\gamma<\mu}$ enumerate all finite subsets of $\mu$ such that  each finite subset  occurs $\mu$ times. In particular, for any $\gamma<\alpha<\mu$ there is some $\alpha<\gamma'<\mu$ for which $u_\gamma=u_{\gamma'}$.

Let $V=\mu$. We define a new graph $G=(V,E)$, extending $G_0$, such that for $\alpha< \beta<\mu$ if $\beta=2\gamma+1$ for some $\gamma<\mu$ and $\alpha\in u_\gamma$ we let $\{\alpha,\beta\}$ be an edge.

We claim that $G$ has the desired properties. 

We note that $G$ is random graph. Indeed, let $X=\{\alpha_1,\dots,\alpha_n\},\,Y=\{\beta_1,\dots,\beta_n\}$ be two disjoint sets of vertices. Let $\gamma<\mu$ be  larger than the $\alpha_i$'s and $\beta_j$'s and satisfying that $u_\gamma=\{\alpha_1,\dots, \alpha_n\}$. Then $2\gamma+1$ is connected to each of the $\alpha_i$ and to none of the $\beta_j$.

Finally, assume that $G$ contains a clique $C$ of cardinality $\aleph_1$. By assumption, $C\cap V_0$ must be at most countable so there must a clique of cardinality $\aleph_1$ consisting of odd ordinals in $\mu$. Let $U$ be the first $\omega$ of those and let $2\gamma+1\in C$ be an ordinal larger than any of the ordinals in $U$. But $2\gamma+1$ can only be connected to finitely many vertices which are smaller, contradiction. 
\end{proof}

\begin{corollary}
For any infinite cardinal $\mu$ there exists a random graph of cardinality and chromatic number $\mu$ with no uncountable clique. 
\end{corollary}
\begin{proof}
Apply Proposition \ref{P:embed into random} to any graph $G_0$ of cardinality $\mu$ with no infinite clique and $\chi(G_0)=\mu$ (for example the triangle-free graph from \cite{ErRa-triangle-free}).

Let $G$ be the graph supplied by the proposition. Since $G_0$ embeds into $G$ and $|G|=\mu$ it follows that $\chi(G)=\mu$.
\end{proof}

\begin{question}
Is there a theory of simple graphs such that for every cardinal $\mu$ we can find a graph of cardinality and chromatic number $\mu$ with no infinite cliques at all?
\end{question}

\begin{question}
Does an analog of Proposition \ref{P:cannot embed infinite complete graph} hold  for other model theoretic tame graphs, such as NSOP$_1$ and NIP?
\end{question}

\subsection{Graphs with Stable Edge Relation}
Before getting into the main result we prove a technical lemma which may be interesting on its own.

\begin{lemma}\label{L:club of models}
Let $T$ be a first order theory, $\mu$ an infinite cardinal and $\varphi(x,y)$ a stable formula. Let $M\vDash T$ with $|M|=\mu^+$ and assume that $M$ is an increasing continuous union of elementary substructures $\langle M_\alpha\rangle_{\alpha<\mu^+}$ each of cardinality at most $\mu$.

Let $\psi(y,z)$ be a uniform definition of $\varphi$-types (as in Fact \ref{F:def of types}). For any $a\in M$ and $\alpha<\mu^+$, let $c_{a,\alpha}\in M_\alpha$ be such that $\psi(y,c_{a,\alpha})$ defines $\tp_\varphi(a/M_\alpha)$.

Then there exists a club $\mathcal{C}\subseteq \mu$ of limit ordinals satisfying that for any $\delta\in \mathcal{C}$ and $a\in M\setminus M_\delta$:

$(\dagger)_{a,\delta}$ For any $\delta<\beta<\mu^+$ there is $b\in M\setminus M_\beta$ for which $\tp_\varphi(b/M_\beta)$ is definable by $\psi(y,c_{a,\delta})$.
\end{lemma}
\begin{proof}
We first prove that  for any $\alpha<\mu^+$, $(\dagger')_\alpha$ there is some limit ordinal $\delta>\alpha$ such $(\dagger)_{a,\delta}$ holds for any $a\in M\setminus M_\delta$. Assume otherwise, i.e. there exists an $\alpha<\mu^+$ for which for any limit ordinal $\delta>\alpha$ there is some $a_\delta\in M\setminus M_\delta$ and $\beta_\delta>\delta$ such that for any $b\in M\setminus M_{\beta_\delta}$, $\tp_\varphi(b/M_{\beta_\delta})$ is not defined by $\psi(y,c_{a_\delta,\delta})$.

For any limit ordinal $\delta>\alpha$, let $f(\delta)$ be the minimal ordinal $\varepsilon$ for which $c_{a_\delta,\delta}\in M_\varepsilon$. Note that as $c_{a_\delta,\delta}$ is a finite tuple and $\delta$ is a limit ordinal, necessarily $f(\delta)<\delta$. As the set of limit ordinals between $\alpha$ and $\mu^+$ is a stationary subset of $\mu^+$ (it is even a club), by Fodor's lemma (\cite[Theorem 8.7]{jech}) there exists a stationary subset $S\subseteq \mu^+$ and $\varepsilon<\mu^+$ for which $f(\delta)=\varepsilon$ for any $\delta\in S$.

By definition, $c_{a_\delta,\delta}\in M_\varepsilon$ for any $ \delta\in S$. As $|M_\varepsilon|\leq \mu$, by the pigeonhole principle there is an unbounded subset $S'\subseteq S$ (i.e. of cardinality $\mu^+$) for which $c:=c_{a_{\delta_1},\delta_1}=c_{a_{\delta_2},\delta_2}$ for any $\delta_1,\delta_2\in S'$.

Now, pick any $\delta\in S'$. By our assumption there is some $\beta_\delta>\delta$ such that for any $b\in M\setminus M_{\beta_\delta}$, $\tp_\varphi(b/M_{\beta_\delta})$ is not defined by $\psi(y,c_{a_\delta,\delta})=\psi(y,c)$.

Let $\beta'>\beta_\delta$ be an element in $S'$ (it is unbounded) and let $a_{\beta'}$ be the corresponding element. So $a_{\beta'}\notin M_{\beta'}$ and in particular $\notin  M_{\beta_\delta}$ and thus $\tp_\varphi(a_{\beta'}/M_{\beta_\delta})$ is not defined by $\psi(y,c_{a_\delta,\delta})=\psi(y,c)$. On the other hand, by choice of $S'$, $\tp_\varphi(a_{\beta'}/M_{\beta'})$ and thus also $\tp_\varphi(a_{\beta'}/M_\beta)$ is defined by $\psi(y,c)$, contradiction. 

Now, we turn to show that we can find such a club $\mathcal{C}$. Let $g:\mu^+\to \mu^+$ be the function mapping $\alpha$ to the minimal limit ordinal $\delta$ satisfying $(\dagger')_\alpha$ and let $\mathcal{C}=\{\alpha<\delta<\mu^+: \text{ $\delta$ is a limit ordinal and } \forall(\alpha<\delta) g(\alpha)<\delta\}$; it is a club in $\mu^+$.

Let $\delta\in \mathcal{C}$ and $a\in M\setminus M_\delta$. As $\delta$ is a limit ordinal, $c_{a,\delta}\subseteq M_\varepsilon$ for some $\varepsilon<\delta$; so $g(\varepsilon)<\delta$. Let $\delta<\beta<\mu^+$; we need to show that there exists $b\in M\setminus M_\beta$ for which $\tp_\varphi(b/M_\beta)$ is defined by $\psi(y,c_{a,\delta})$.

By $(\dagger)_{a,g(\varepsilon)}$, since $a\notin M_{g(\varepsilon)}$ and $\beta>\delta>g(\varepsilon)$, there is $b\in M\setminus M_\beta$ for which  $\tp_\varphi(b/M_\beta)$ is defined by $\psi(y,c_{a,\delta})$, as required.
\end{proof}

We phrase forking symmetry for a stable formula in a form useful to us.
\begin{fact}[Harrington]\label{F:local Harrington}
Let $T$ be a first order theory with monster model $\mathbb{U}$. Let $\varphi(x,y)$ be a stable formula and $\psi(y,z)$ a formula uniformly defining $\varphi$-types. For any two small models $N_1$ and $N_2$ and elements $a,b\in \mathbb{U}$, if $\psi(y,c_a)$ defines $\tp_\varphi (a/N_1)$ and $\psi(y,c_b)$ defines $\tp_\varphi(b/N_2)$ and $c_a\in N_2,\, c_b\in N_1$ then 
\[\psi(a,c_b)\Leftrightarrow \psi(b,c_a).\]
\end{fact}
\begin{proof}
Let $p=\tp_\varphi(a/N_1)$ and $q=\tp_\varphi(b/N_2)$.
Let $\widetilde p\supseteq p$ be the global $\varphi$-type that $\psi(y,c_a)$ defines and $\widetilde q\supseteq q$ be the global $\varphi$-type that $\psi(y,c_b)$ defines. By \cite[Lemma 8.3.4]{TZ}, $\psi(x,c_b)\in \widetilde p\Leftrightarrow \psi(x,c_a)\in \widetilde q$; but as $c_a\in N_2$ and $c_b\in N_1$, we conclude.
\end{proof}

The following is due to Engelking-Kar\l owicz \cite{En-Ka}, see also \cite{Rinot}.
\begin{fact}\label{F:En-Ka}
For cardinals $\kappa\leq \lambda\leq \mu\leq 2^\lambda$ the following are equivalent:
\begin{enumerate}
\item $\lambda^{<\kappa}=\lambda$
\item there exists a collection of functions $\langle f_i:\mu\to \lambda\rangle_{i<\lambda}$, such that for every $X\in [\mu]^{<\kappa}$ and every function $f:X\to \lambda$, there exists some $i<\lambda$ with $f\subseteq f_i$.
\end{enumerate}
\end{fact}

We now prove the main result of this section.

\begin{proposition}\label{P:stable}
Let $L=\{E\}$ be the language of graphs and $T$ be an $L$-theory specifying that $E$ is a symmetric and irreflexive stable relation. For any infinite cardinal $\mu$ and $G\vDash T$ with $|G|=\mu^+$, if $\chi(G)\geq \mu^+$ then $G$ contains an infinite clique of cardinality $\mu^+$.
\end{proposition}
\begin{proof}
For ease of notation, write $\varphi(x,y)=E(x,y)$ and let $\psi(y,z)$ be a uniform definition for $\varphi$-types (as in Fact \ref{F:def of types}). Assume that $\chi(G)\geq \mu^+$.

By Lemma \ref{L:club of models}, there exists an increasing  continuous family of elementary substructures $\{G_\alpha\prec G:0<\alpha<\mu^+\}$   with $|G_\alpha|=\mu$ and $\bigcup_{0<\alpha<\mu^+} G_\alpha=G$, satisfying:

$(\dagger)$ For any $0<\delta<\mu^+$ and $a\in G\setminus G_\delta$  there exists $c_{a,\delta}\in G_\delta$ such that $\psi(y,c_{a,\delta})$ defines $\tp_\varphi(a/G_\delta)$ satisfying that  for any $\delta<\beta<\mu^+$ there is $b\notin G_\beta$ for which $\tp_\varphi(b/G_\beta)$ is defined by $\psi(y,c_{a,\delta})$. 

Set $G_0=\emptyset$

For any $a\in G$ let $\alpha^a_0<\mu^+$ be minimal such that $a\in G_{\alpha_0+1}$. Let $c_0^a=c_{a,\alpha_0^a}\in G_{\alpha^a_0}$, in particular $\psi(y,c_0^a)$ defines $\tp_\varphi (a/G_{\alpha^a_0})$. Let $\alpha_1^a<\alpha_0^a$ be such that $c_0^a\in G_{\alpha_1^a+1}\setminus G_{\alpha^a_1}$. Let $c^a_1=c_{a,\alpha_1^a}$. Likewise we continue and find a sequence $\langle (c_{k-1}^a,\alpha_k^a):1\leq k\leq n_a\rangle$ satisfying
\begin{list}{•}{}
\item $c^a_{k-1}=c_{a,\alpha^a_{k-1}}$, in particular $\psi(y,c^a_{k-1})$ defines $\tp_\varphi(a/G_{\alpha^a_{k-1}})$
\item $c^a_{k-1}\in G_{\alpha^a_k+1}\setminus G_{\alpha^a_k}$. 
\item $\alpha^a_{n_a}=0$.
\end{list}

Note that $\alpha_0=0$ if and only if $n_a=0$ and then there are no $c$'s.

We now define a coloring.

By  Engelking-Kar\l owicz (Fact \ref{F:En-Ka}), there exists a family of functions $\{g_{\beta}:\mu^+\to \mu:\beta<\mu\}$ satisfying that for any finite subset $X\subseteq \mu^+$ and every function $f:X\to \mu$ there is some function $g_\beta$, $\beta<\mu$, with $f\subseteq g_\beta$.

For any such $a\in G$ define a function $f_a:\{\alpha^a_0,\dots,\alpha_{n_a}^a\}\to \omega$ by setting $f_a(\alpha^a_i)=i$. Let $\beta^a<\mu$ be minimal such that $f_a\subseteq g_{\beta^a}$.

For any $\alpha<\mu^+$ let $s_\alpha :G_\alpha\to \mu$ be a bijection. For simplicity, we also denote by $s_\alpha$ the induced bijection between $G_\alpha^n$ and $\mu^n$. 
Let $Y=\aleph_0 \times \mu^2\times \mu^{<\omega}$; note that $|Y|=\mu$.
%

Let $\phi:G\to Y$ be a function mapping $a\in G$ to \[(n_a, \beta^a,s_{\alpha^a_0+1}(a),(s_{\alpha_i^a+1}(c^a_{i-1}))_{1\leq i\leq n_a}).\] Since $\phi$ cannot be a legal coloring there exist $a,b\in G$ distinct elements satisfying that $\phi(a)=\phi(b)$ and $\varphi(a,b)$ holds, i.e. $a\mathrel{E} b$.

%
%

Without loss of generality, assume that $\alpha^a_0\geq \alpha^b_0$. Note that necessarily, $\alpha^a_0>\alpha^b_0$, since otherwise, as $s_{\alpha_0^a+1}(a)=s_{\alpha_0^b+1}(b)$ we conclude that $a=b$, contradiction. In particular, $n_a>0$ and so also $n_b=n_a>0$.

Also, if $\alpha_i^a=\alpha^b_j$ for some $i,j$ then, since by assumption $\beta^a=\beta^b$, $i=j$.

Let $n>0$ be minimal such that $\alpha_n^a=\alpha^b_n$ (such $n$ exists since this equality holds for $n=n_a>0$). Consider the (ordered) set $A'=\{\alpha_i^a,\alpha_i^b:i< n\}$. Note that the elements in $A'$ are distinct. Now let $A=A'\cup\{\alpha^a_n=\alpha^b_n\}$.

We will call an element $\alpha_i^a\in A$ (with $i>0$) an \emph{$a$-pivot} if there exists an element $\alpha_j^b\in A$ with $\alpha^a_{i-1}>\alpha_j^b>\alpha_i^a$ (and likewise a \emph{$b$-pivot} for $\alpha_i^b\in A$ with $i>0$). Note that $\alpha^a_n=\alpha^b_n$ is either an a-pivot or a b-pivot.

We will prove the following by (downward) induction.
\begin{claim}
For every $a$-pivot $\alpha_i^a\in A$, with $i>0$, $\psi(b,c^a_{i-1})$ holds and for every $b$-pivot $\alpha_i^b\in A$, with $i>0$, $\psi(a,c^b_{i-1})$ holds.
\end{claim}
\begin{claimproof}
%
%
%
Let $\alpha^a_{t+1}\in A$ be an $a$-pivot (possibly $t+1=n$). Let $v<n$ be minimal with $\alpha^a_t>\alpha^b_{v}>\alpha^a_{t+1}$; $\alpha^b_v$ is either a $b$-pivot and $v=s+1$ for some $s$ or $v=0$. Assume, first, that $\alpha^b_{s+1}$ is a $b$-pivot; so $\psi(a,c^b_s)$ holds by the induction hypothesis. As $c^b_s\in G_{\alpha^b_{s+1}+1}\subseteq G_{\alpha_t^a}$ and $c^a_t\in G_{\alpha^a_{t+1}+1}\subseteq G_{\alpha^b_{s+1}}\subseteq G_{\alpha^b_s}$, we can apply Fact \ref{F:local Harrington} with $\tp_\varphi(a/G_{\alpha^a_t})$ and $\tp_\varphi(b/G_{\alpha^b_s})$ and conclude that $\psi(b,c^a_t)$ holds.  

Now assume that $v=0$. Since $b\in G_{\alpha^b_0+1}\subseteq G_{\alpha^a_t}$ and $\varphi(a,b)$ holds then, as $\psi(y,c^a_t)$ defines $\tp_\varphi (a/G_{\alpha^a_t})$ we get that $\psi(b,c^a_t)$ holds; as needed.

The case where $\alpha^b_{t+1}\in A$ is a $b$-pivot is proved similarly.
\end{claimproof}

Note that $c^a_{n-1}\in G_{\alpha^a_n+1}=G_{\alpha^b_n+1}\ni c^b_{n-1}$. As $\phi(a)=\phi(b)$, we necessarily have $c:=c^a_{n-1}=c^b_{n-1}$.
 If $\alpha^a_n$ is an a-pivot we have that $\psi(a,c)$ holds and if it is a b-pivot then $\psi(b,c)$ holds. Assume the former holds (the proof where the latter holds is identical).
 
%
%

We inductively construct a sequence $(d_\alpha)_{\alpha<\mu^+}$ in $V$ such that:
\begin{list}{•}{}
\item $\psi(d_\alpha,c)$ holds for all $\alpha<\mu^+$.
\item $\varphi(d_\alpha,d_\beta)$ for all $\alpha\neq \beta<\mu^+$.
\end{list}

Suppose we constructed $(d_\alpha)_{\alpha<\gamma}$. Let $\alpha_{n-1}^a<\delta<\mu^+$ be such that $d_\alpha\in G_\delta$ for any $\alpha<\gamma$. By $(\dagger)$ there exists $d_\gamma\in G\setminus G_\delta$ such that $\tp_\varphi(d_\gamma/G_\delta)$ is definable by $\psi(x,c)$.

Since $\psi(d_\alpha,c)$ holds for any $\alpha<\gamma$ it follows that $\varphi(d_\gamma,d_\alpha)$ holds (and also $\varphi(d_\alpha,d_\gamma)$ by symmetry). Additionally, $\psi(x,c)$ defines both $\tp_\varphi(a/G_{\alpha^a_{n-1}})$ and $\tp_\varphi(d_\gamma/G_{\alpha^a_{n-1}})$; hence $a\equiv^\varphi_{G_{\alpha^a_{n-1}}} d_\gamma$.  As $c\in G_{\alpha^a_{n-1}}$ and $\psi(x,c)$ is equivalent to a boolean combination of instances of $\varphi^{\text{opp}}$-formulas over $G_{\alpha^a_{n-1}}$, it follows by symmetry of $\varphi$ that $a\equiv^{\varphi^{\text{opp}}}_{G_{\alpha^a_{n-1}}} d_\gamma$ so $\psi(d_\gamma,c)$ holds as well.
%
%
%
%
\end{proof}

Given Proposition \ref{P:stable} (and Proposition \ref{P:cannot embed infinite complete graph}), it is natural to ask under which conditions on the cardinality of $G$ does the  proposition hold. The following example shows that it fails for strong limit cardinals.

\begin{proposition}\label{P: stable no arbt large finite cliques}
Let $\lambda$ be a strong limit cardinal.\footnote{That means that $2^\mu<\lambda$ for any $\mu<\lambda$. Any such cardinal is a limit cardinal. An example is $\beth_{\omega}(\aleph_0)$.} There exists a non-stable graph with stable edge relation $G$ of cardinality $\lambda$ with $\chi(G)\geq \lambda$ for which we cannot embed arbitrary large finite cliques. In fact, it is triangle-free.
\end{proposition}
\begin{proof}
For any $\mu<\lambda$, let $G_\mu=\Sh_2(\beth_1(\mu)^+)$. By Lemma \ref{L:properties of shift2}, it is not stable but has a stable edge relation, and it is triangle-free. By Fact \ref{F:Sh-high chrom}, $\chi(G_\mu)\geq \mu^+$.

Let $G=\bigoplus_{\mu<\lambda}G_\mu$ be the direct sum of all of these graphs. Thus $\chi(G)\geq \lambda$ and $|G|=\lambda$. The graph $G$ is not stable but its edge relation is stable.  On the other hand, since each of the $G_\mu$'s is triangle-free,  we cannot embed arbitrary large finite cliques into $G$.
\end{proof}

\begin{remark} \label{R:even stable theory}
Using Remark \ref{R:Sh3}(2), we may replace $\Sh_2(\beth_1(\mu)^+)$ by $\Sh_3^{sym}(\beth_2(\mu)^+)$ and arrive at a stable graph with the prescribed properties.
\end{remark}

\section{The Chromatic Spectrum}\label{S:chromatic}
We rephrase the results of the previous section in a different manner.

For $T$ a theory of graphs and $\mu$ an infinite cardinal, let \[\ch(\mu)=\min\{ |G|: G\vDash T,\, \chi(G)\geq \mu\}.\]

We employ the convention that $\min\emptyset=\infty$. Note that if $\ch(\mu)=\infty$ then $\ch(\lambda)=\infty$ for all $\lambda\geq \mu$.

\begin{remark}
\begin{enumerate}
\item Since for any graph $G$, $\chi(G)\leq |G|$, $\ch(\mu)\geq \mu$.
\item Let $T=\mathrm{Th}(\Sh_k(\omega))$. By Fact \ref{F:Sh-high chrom}, $\ch(\lambda^+)\leq \beth_{k-1}(\lambda)^+$. On the other hand, if towards a contradiction we assume that $\ch(\lambda^+)<\beth_{k-1}(\lambda)^+$ then we can find $M\vDash T$ with $\chi(M)\geq \lambda^+$ and $|M|\leq \beth_{k-1}(\lambda)$. By L\"owenheim-Skolem we can can assume that $|M|=\beth_{k-1}(\lambda)$. By Lemma \ref{L:embed model of shift into symmetric}, we may embed $M$ into $\Sh_n^{sym}(\beth_{k-1}(\lambda))$, so by Fact \ref{F:Sh-high chrom} $\chi(M)\leq \lambda$, contradiction. We conclude that $\ch(\lambda^+)=\beth_{k-1}(\lambda)^+$.
\end{enumerate}
\end{remark}

We can now rephrase the main result of \cite{1211} using this function:

\begin{proposition}
The following are equivalent for a stable theory of graphs  $T$:
\begin{enumerate}
\item There is some $G\vDash T$ and a natural number $k$ such that $G$ contains $\Sh_k(n)$ for all $n$.
\item $\ch(\beth_2(\aleph_0)^+)< \infty$.
\item For any cardinal $\mu$,  $\ch(\mu)< \infty$.
\item For any  cardinal $\mu$, $\ch(\mu)< \beth_\omega(\mu)$.
\end{enumerate}
\end{proposition}
\begin{proof}
$\neg (1)\implies \neg (2)$. By the main theorem of \cite[Corollary 6.2]{1211}, $\chi(G)\leq \beth_2(\aleph_0)$ for all $G\vDash T$, i.e. $\ch(\beth_2(\aleph_0)^+)=\infty$.

$(4)\implies (3)\implies (2)$. Easy.

$(1)\implies (4)$.  Let $\mu$ be an infinite cardinal. By compactness and (1) we can embed $\Sh_k(\beth_{k-1}(\mu)^+)$ in a large enough model of $T$. So by L\"owenheim-Skolem we can find a model $G$ of cardinality $\beth_{k-1}(\mu)^+$ with $\chi(G)\geq \chi\left(\Sh_k(\beth_{k-1}(\mu)^+\right)\geq \mu^+\geq \mu$ (using Fact \ref{F:Sh-high chrom}). Consequently, $\ch(\mu)\leq\beth_{k-1}(\mu)^+< \beth_\omega(\mu)$.
\end{proof}

Next, we phrase the results of the previous section for simple graphs and graphs with stable edge relation using $\ch$. 

\begin{proposition}\label{P:chi for simple and stable}
Let $T$ be a theory of graphs and assume that either $T$ is simple or the edge relation is stable. The following are equivalent:
\begin{enumerate}
\item $T$ proves the existence of arbitrary large finite cliques.
\item for any infinite cardinal $\mu$, $\ch(\mu)= \mu$.
\item for any infinite cardinal $\mu$, $\ch(\mu^+)= \mu^+$.
\item there exists an infinite cardinal $\mu$ for which $\ch(\mu^+)=\mu^+$.
\end{enumerate}
If $T$ is simple then they are also equivalent to:
\begin{enumerate}
\item[(5)] there exists an infinite cardinal $\mu$ with $\ch(\mu^+)\leq 2^\mu$.
\item[(6)] for any infinite cardinal $\mu$,  $\ch(\mu^+)\leq 2^{\mu}$.
\end{enumerate}

\end{proposition}
\begin{proof}
$(1)\implies (2)$.  By compactness and L\"owenheim-Skolem, there is a model $G$ of cardinality  $\mu$ which has an infinite clique of cardinality $\mu$. Thus $\chi(G)\geq \mu$, so $\ch(\mu)=\mu$.

 $(2)\implies (3)\implies (4)$ is easy.
 
$(4)\implies (1)$.  Assume that $(4)$ holds and let $\mu$ be an infinite cardinal for which $\ch(\mu^+)=\mu^+$. Thus there exists a model $G$ with $\chi(G)\geq \mu^+$ and $|G|= \mu^+$. By Proposition \ref{P:cannot embed infinite complete graph} for the simple case and Proposition \ref{P:stable} for stable edge relation case, $G$ contains arbitrary large finite cliques.

Assume that $T$ is simple.

$(4)\implies (5)$ is easy and $(5)\implies (1)$ uses Proposition \ref{P:cannot embed infinite complete graph} as above. 

$(3)\implies (6)\implies (5)$ is easy.
\end{proof}

Is there an analog to Proposition \ref{P:chi for simple and stable} for general shift graphs? A reasonable suggestion is:
\begin{conjecture}
  Suppose that $T$ is stable. If for all cardinals $\mu$, $\ch(\mu^+) \leq \beth_{n-1}(\mu)^+$ then for some $m\leq n$, there is an embedding of $\Sh_{m}(\omega)$ in any $\omega$-saturated model of $T$.
\end{conjecture}
Note this is exactly Question \ref{Q:The question on Ch} without assuming GCH.

\appendix
\section{An Example by Hajnal and Komj\'{a}th} \label{A:Hajnal-Komjath example}

In this section we present an example due to Hajnal and Komj\'{a}th \cite[Theorem 4]{HK}. This is an example of a graph of size continuum whose chromatic number is $\aleph_1$, which does not contain all finite subgraphs of any shift graph $\Sh_n(\omega)$. They gave it as an example refuting Taylor's strong conjecture (which does hold outright for $\omega$-stable graphs with a close relative of it holding for stable theories in general by \cite{1196,1211}). The main goal here is to prove that this example has the independence property (IP) (thus is not stable) and furthermore that its theory is not simple.

\begin{definition}\label{D:special}
  A graph $G=(V,E)$ is called \emph{special} if there exists partial order $\prec$ on $V$ satisfying that:
  \begin{enumerate}
    \item if $x\E y$ then either $x\prec y$ or $y\succ x$ and
    \item there is no circuit $C=\langle x_0,\dots,x_{n-1}\rangle$, $n\geq 3$, of the form \[x_0\prec x_1\prec \dots\prec x_{m-1}\prec x_m\succ x_{m+1}\succ \dots\succ x_{n-1}\succ x_0.\]
  \end{enumerate} 
\end{definition}

\begin{proposition}\label{P:special has no shift}\cite[Theorem 4]{HK}
Let $G=(V,E)$ be a special graph as witnessed by $\prec $. Then for all $n\geq 1$, $G$ does not contain all the finite subgraphs of $\Sh_n(\omega)$.
\end{proposition}
\begin{proof}
Assume towards a contradiction that $G$ contains all finite subgraphs of $\Sh_n(\omega)$ for some $n\geq 1$. If $n=1$ then it must contain a triangle so obviously contradicts Definition \ref{D:special}(1); so we assume that $n\geq 2$. Coloring pairs of $<$-increasing tuples, by Ramsey there is some integer $r$ such that if $f:\Sh_n(r)\to G$ is an embedding, there is $A\subseteq r$, $|A|=2n+1$ such that either $f(a_0,\dots,a_{n-1})\prec f(a_1,\dots,a_n)$ for all strictly increasing $n+1$-tuples from $A$ $(a_0,\dots,a_n)$ or $f(a_0,\dots,a_{n-1})\succ f(a_1,\dots,a_n)$ for all strictly increasing $n+1$-tuples from $A$, $(a_0,\dots,a_n)$. 

Assume the former occurs and that for simplicity $A=\{0,\dots,2n\}$. Then
\[(0,\dots,n-1)\prec (1,\dots,n-1,n+1)\prec\dots\prec (n-1,n+1,\dots, 2n-1)\prec (n+1,\dots,2n)\]
and
\[(0,\dots,n_1)\prec (1,\dots,n-1,n)\prec\dots\prec (n,n+1,\dots,2n-1)\prec (n+1,\dots,2n),\]
which is a contradiction.
\end{proof}

%

\begin{proposition}\label{P:Hajnal-Komjath Example}\cite[Theorem 4]{HK}
There exists a graph $G=(V,E)$, with $|V|=2^{\aleph_0}$ satisfying the following properties:
\begin{enumerate}
\item $G$ is special and in particular for every $n\geq 1$ it does not contain all finite subgraphs of $\Sh_n(\omega)$.
\item $\chi(G)=\aleph_1$.
\item $G$ has IP and in particular is not stable (in fact the edge relation has IP).
\item $G$ is not simple.
\end{enumerate}
\end{proposition}
\begin{proof}
Let $V=\{T_\alpha:\alpha<\aleph_1\}$ be a collection of disjoint sets with $|T_\alpha|=2^{\aleph_0}$ for each $\alpha<\aleph_1$. We define an edge relation on $V$ turning it to a graph satisfying our desired properties. 

To define the edge relation we define for $x\in T_\alpha$ and $\alpha<\aleph_1$, $G(x)=\{y\in T_{<\alpha}: x\E y\}$ by induction on $\alpha$.

If $\alpha=\beta+1$ we let $G(x)=\emptyset$ for every $x\in T_\alpha$ so assume that $\alpha$ is a limit ordinal and that $G(x)$ has already been defined for $x\in T_{<\alpha}=\bigcup_{\beta<\alpha} T_\alpha$.

For $\gamma<\alpha$ and $y\in T_{<\alpha}$, we say that \emph{$y$ is $\gamma$-covered} if there exists $\alpha_0<\dots<\alpha_m$, with $\alpha_0\leq \gamma$, and $x_i\in T_{\alpha_i}$ with $x_m=y$ such that $x_0\E x_1\dots \E x_m$.  Note that any $y\in T_\gamma$ is $\gamma$-covered as witnessed by the trivial path.

Let $\mathcal{W}_\alpha$ be the collection of all subsets $W\subseteq T_{<\alpha}$ satisfying that
\begin{list}{•}{}
\item $W=\{x_n: n<\omega\}$ is countable
\item $x_n\in T_{\alpha_n}$ and $\alpha_n<\alpha_m$ whenever $n<m<\omega$.
\item $\sup\{\alpha_n:n<\omega\}=\alpha$.
\item no $x_n$ is $\alpha_{n-1}$-covered for $0<n<\omega$.
\end{list}
Obviously, $|\mathcal{W}_\alpha|\leq 2^{\aleph_0}$; choose some enumeration $\mathcal{W}_\alpha=\{W_\gamma:\gamma<2^{\aleph_0}\}$ and $T_\alpha=\{t_\gamma:\gamma<2^{\aleph_0}\}$ and set $G(t_\gamma)=W_\gamma$. I.e for $y\in T_{<\alpha}$ and $x=t_\gamma\in T_\alpha$, $x\E y\iff y\in W_\gamma$. Let $G=(V,E)$.

We show that $G$ satisfies the properties listed in the statement.

We first show (1). Let $C=\langle x_0,\dots, x_{n-1}\rangle\subseteq V$, $n\geq 3$, be a circuit in $G$ with $x_i\in T_{\alpha_i}$ and
\[\alpha_0<\alpha_1<\dots<\alpha_{m-1}<\alpha_m>\alpha_{m+1}>\dots>\alpha_{n-1}>\alpha_0.\] for some $0<m<n-1$. If $\alpha_{m-1}= \alpha_{m+1}$, then as the elements of $C$ are distinct, $x_m$ would be connected to two vertices in $T_{\alpha_{m-1}}=T_{\alpha_{m+1}}$, contradicting our construction. So assume without loss of generality that $\alpha_{m-1}<\alpha_{m+1}$. Thus $x_{m-1},x_{m+1}\in G(x_m)$ and $x_{m+1}$ is $\alpha_{m-1}$-covered as witnessed by $\alpha_0<\alpha_{n-1}<\dots<\alpha_{m+1}$ and $\alpha_{m-1}\geq \alpha_0$; on the other hand since $\alpha_{m-1}<\alpha_{m+1}$ we get a contradiction.

We show (2). Let $c:G\to \aleph_0$ be a coloring. We say that a color $n<\omega$ is \emph{small} if there is a $\gamma_n<\aleph_1$ such that every point $x\in V$ with $c(x)=n$ is $\gamma_n$-covered; otherwise, call $n$ \emph{large}. For any small $n<\omega$ choose $\gamma_n$ minimal satisfying the above. Put $\gamma=\sup\{\gamma_n : n\text{ small}\}<\aleph_1$. We note that there exist large $n$; indeed take any $x\in T_{\gamma+1}$ and let $n<\omega$ be with $c(x)=n$. Since $x$ is not connected to any $y\in T_{<\gamma+1}$ it cannot be $\gamma$-covered. 

If $n<\aleph_0$ is large then for every $\alpha<\aleph_1$ there exists $x\in V$ with $c(x)=n$ which is not $\alpha$-covered. 

Let $\langle m_i<\aleph_0 :i<\omega\rangle$ be a sequence (possibly with repetitions), containing all large colors.

By definition, there exists $x_0\in T_{\alpha_0}$ with $c(x_0)=m_0$ which is not $\gamma$-covered (so necessarily $\gamma<\alpha_0$). We continue inductively and for any $n<\omega$ we find $x_n\in T_{\alpha_n}$ with $c(x_n)=m_n$ which is not $\alpha_{n-1}$-covered (so necessarily $\alpha_{n-1}<\alpha_n$). 

Let $\alpha=\sup\{\alpha_n:n<\omega\}$, it is necessarily a limit ordinal, and let $W=\{x_n :n<\omega\}\subseteq T_{<\alpha}$. By definition, there exists an element $x\in T_\alpha$ with $G(x)=W$. In particular $c(x)\neq m_n$ for all $n<\omega$, so $c(x)=k$ is small, i.e. it is $\gamma_k$-covered. But by the definition above, any $y\in G(x)=W$ is not $\gamma$-covered; so it cannot be that $x$ is $\gamma_k\leq \gamma$-covered.

To show that $\chi(G)=\aleph_1$ note that $c:V\to \aleph_1$ defined by $c(x)=\alpha_x$ for $x\in T_{\alpha_x}$ is a legal coloring.

%
%

To show (3), choose for each $n<\omega$ some $x_n \in T_n$. Then, for every unbounded subsets $W\subseteq \omega$ there exists a unique $x\in T_\omega$ with $G(x)=\{x_n : n \in W\}$, giving IP.

Item (4) is a direct consequence of Proposition \ref{P:cannot embed infinite complete graph} (with $\mu=\aleph_0$).
\end{proof}
\begin{question}
Can one find such a counterexample which is stable? simple?
\end{question}

\bibliographystyle{alpha}
\bibliography{f2059}
\end{document}